\documentclass[12pt,a4paper]{article}
\usepackage{amssymb,amsmath,enumerate,amsthm}
\usepackage[mathcal]{euscript}

 \theoremstyle{plain}
 \newtheorem{Theorem}{Theorem}
 \newtheorem{Corollary}{Corollary}
 \newtheorem{Lemma}{Lemma}
 \newtheorem{Proposition}{Proposition}
 \theoremstyle{definition}
 \newtheorem{Definition}{Definition}

\newcommand{\nr}{\nonumber}
\newcommand{\gr}{\textsl{grad\;}}
\newcommand{\s}{\mbox{\small{$\mathcal{S}$}}}

\newcommand{\sa}{\mbox{\scriptsize{$\mathcal{S}$}}}

\title{A class of  semibasic vector 1\,--\,forms\\ on   Finsler manifolds}
\author{A. Tayebi and M. Barzegari}

\begin{document}

\maketitle

\begin{abstract}
In this paper, we define  conservative semibasic vector $1-$forms on the tangent bundle of a Finsler manifold. Using these vector $1-$forms, we characterize conservative $L-$Ehresmann connections with respect to the energy function. Then we find  a correspondence between torsion-free semibasic vector $1-$forms and the subset of  vertical vector fields. Taking into account this correspondence,  we construct a class of semisprays that generates the Ehresmann connections mentioned above.\\\\
{\bf Keywords:} Finsler manifold, Ehresmann connection, Semibasic vector 1-form.\footnote{ 2010 Mathematics subject Classification:  53C07, 53C60.}

\end{abstract}

\section{Introduction}
The difference of  two Ehresmann connections over a manifold $M$ is a semibasic vector $1-$forms on $TM$.
Motivated by this fact, it is natural to define a conservative semibasic vector $1-$form on a Finsler manifold with respect to an energy function induced by a Finsler metric. A semibasic vector $1-$form is called conservative on a Finsler manifold $(M,E)$ if it can be expressed as a difference of two conservative Ehresmann connections on $M$. The conservativity property of a semibasic vector $1-$form $L$ on $TM$ implies the conservativity of an $L-$Ehresmann connection, which is defined by
\[\mathbf{h}_{L}:=\mathbf{h}_{\,0}+L+[J,(d_LE)^{\#}],\]
where $\mathbf{h}_{\,0}$ is the Berwald connection. It is well known that this class of Ehresmann connections contains  the Wagner connection of the Finsler manifold $(M,E)$ and among them the conservative ones are conformally closed. For more details see \cite{RB} and \cite{MBN}.

In Section 3, using conservative semibasic vector $1-$forms, we characterize conservative $L-$Ehresmann connections with respect to the energy function $E$.

The weak torsion of an $L-$Ehresmann connection is in the following form $$\mathbf{t}_L=[J,L].$$ Thus, it is natural to define torsion-freeness of a semibasic vector $1-$form $L$ on the tangent manifold $TM$ by requiring the condition $[J,L]=0.$ In Section 4, we establish  a correspondence between the subset of  vertical vector fields $\mathfrak{X}^v(TM)$
and the set of torsion\,--\,free semibasic vector $1-$forms on $TM$.

Taking into account this correspondence, in Subsection 4.2, we introduce a set of semisprays as follows
\[
\overset{\,\,V}{\s}=\s_0+2\,V+2\,(d_{\,[J,V]}E)^\#,
\]
where $V$ is a vertical vector field on $TM$. If $V$ is a two-homogeneous vertical vector field, then the semispray $\overset{\,\,V}{\s}$ is a spray and we can find projectively related relation between two sprays in this form.

In \cite{VIN}, Vincze presents the theory of conservative semisparays on a Finsler
manifold $(M,E)$ with respect to the energy function $E$. Then he defines a conservative vertical vector field on a Finsler manifold. In this paper,  using conservativity and torsion-freeness of the $L-$Ehresmann connections, we find a class of conservative vertical vector fields on a Finsler manifold.

\section{Preliminaries}

We work on an $n-$dimensional connected smooth manifold $M$ whose topology is Hausdorff and has a countable base. $C^\infty(M)$
denotes the ring of smooth real-valued functions on $M$, $\mathfrak{X}(M)$ and  $\Omega^k(M)$ stand for the $C^\infty(M)-$ module of (smooth)
vector fields  and differential $k-$forms on $M$.

$TM$ is the tangent manifold of $M$, and $\overset{\circ}{T}M$ is the open submanifold of the non-zero tangent vectors to $M$.
The vertical and the complete lift of a smooth function $f$ on $M$ into $TM$ are denoted by $f^v$ and $f^c$, respectively.

The $C^\infty (TM)-$module of vertical vector fields on $TM$ will
be denoted by $\mathfrak{X}^v(TM)$. $X^v$  stands for the vertical lift of a vector field $X$ on $M$ and $C\in\mathfrak{X}^v(TM)$ is the \emph{Liouville  vector field}.

\bigskip

By a \emph{vector $k-$form} on $TM$ we mean a skew symmetric
$C^\infty(TM)-$ multilinear map $K:(\mathfrak{X}(TM))^k\rightarrow
\mathfrak{X}(TM)$ if $k\in \{1,\ldots,2n\}$, and a vector field on
$TM$, if $k=0$. In particular, a vector $1-$form on $TM$ is just a
tensor field of type $(1,1)$. The $C^\infty(TM)-$module of vector
$k-$forms on $TM$ will be denoted by $\Psi^k(TM)$.

\bigskip
By the \emph{Fr\"{o}licher-Nijenhuis theory} of vector forms to any
vector $k-$form $K\in\Psi^k(TM)$ two graded derivations of
$\Omega(TM)$ are associated, denoted by $i_K$ and $d_K$, which the former is of
degree $k-1$, and the later is of degree $k$, and the following rules are
prescribed:
\begin{align}
&i_K\upharpoonright C^\infty(TM)=0;\quad i_K \circ\alpha=\alpha\circ
K,\quad \textrm{if} \,\,\alpha \in \Omega^1(TM);\label{d1}
\\
&d_K:=[i_K,d]=i_K\circ d- (-1)^{k-1}d \circ i_K.\label{d2}
\end{align}
Then, in particular,
\begin{align}
d_K\varphi=d\varphi\circ K;\qquad \varphi\in C^\infty (TM)\ ,\ K\in\Psi^k(TM).\label{d3}
\end{align}
To any vector forms $K\in\Psi^k(TM)$, $L\in\Psi^\ell(TM)$ there is a unique vector
$(k+l)-$form $[K,L]\in\Psi^{k+l}(TM)$, the \emph{Fr\"{o}licher-Nijenhuis bracket} of $K$ and $\L$ such that
\begin{align}\label{e}
d_{[K,L]}=[d_K,d_L].
\end{align}
In particular, if $K\in \Psi^1(TM)$, $Y\in\Psi^\circ(TM)$ then $[K,Y]\in\Psi^1(TM)$,
and for any vector field $X\in\mathfrak{X}(TM)$, we have,
\begin{gather}
[K,Y]X=[KX,Y]-K[X,Y],\\
i_{\,[K,Y]}=i_Y\circ d_K + d_K\circ i_Y -\mathcal{L}_{KY},\label{ee2}\\
i_Y\circ i_K=i_K\circ i_Y+i_{KY}.\label{new1}
\end{gather}
There is a unique vector $1-$form $J$ on $TM$ such that
\begin{gather}\label{ee4}
Im\,J=Ker\,J=\mathfrak{X}^v(TM),
\quad [J,J]=0,\quad
[J,C]=J.
\end{gather}
$J$ is called the \emph{vertical endomorphism}. A differential form $\alpha \in \Omega ^k(TM)$ is \emph{semibasic}, if
$i_{J\xi}\alpha=0;$ a vector form $K\in\Psi^k(TM)$ is \emph{semibasic}, if $i_{J\xi}K=0$ and
$J\circ K=0$ $(k \geq 1,\ \xi\in \mathfrak{X}(TM))$.

 \bigskip
 A vector field $\s$  on $TM$ of class $C^1$, smooth on  $\overset{\circ}{T}M$ is said to be a \emph{semispray} over $M$ if $J\s=C$.  A semispray $\s$ is called a \emph{spray} if it is homogeneous of degree 2, i.e., $[C,\s]=\s$.

A vector 1--form $\mathbf{h}\in \Psi^{1}(TM)$, smooth --in general--
only over $\overset{\circ}{T}M$ is said to be a \emph{Ehresmann connection}  over $M$ if it is a
projector (i.e., $\mathbf{h}^2=\mathbf{h}$) and
$\ker\mathbf{h}=\mathfrak{X}^v(TM)$, or, equivalently, if $J\circ\mathbf{h}= J$ and $\mathbf{h}\circ J=0$.
An Ehresmann connection $\mathbf{h}$ is called \textit{homogeneous} if its tension vanishes, i.e., $\mathbf{H}=[C,\mathbf{h}]=0$. The \emph{weak} \emph{torsion} of
$\mathbf{h}$ is the vector 2-form $ \mathbf{t}:=[J,\mathbf{h}]$.

\medskip

A fundamental result due to M. Crampin and J. Grifone states that
\emph{any semispray $\s$ generates a Ehresmann connection of zero
weak torsion} by the formula
\begin{equation}\label{eh}
\mathbf{h}=\frac{1}{2}\,(1_{\mathfrak{X}(TM)}+[J,\s]).
\end{equation}

\bigskip

Let a function $E:TM\rightarrow \mathbb{R}$ be given. Assume:
\begin{itemize}
    \item[(i)]
    \ $E(v)>0$ for all $v\in \overset{\circ}{T}M,\,\,\, E(0)=0$;
    \item[(ii)]
    \ $E$ is of class $C^{1}$ on $TM$, smooth on
    $\overset{\circ}{T}M$;
    \item[(iii)]
    \ $E$ is (positive--) homogeneous of degree 2, i.e., $CE=2E$;
    \item[(iv)]
    \ The \emph{fundamental 2-form} $\omega:=d\,d_JE$ is nondegenerate.
\end{itemize}
Then $(M,E)$ is said to be a \emph{Finsler manifold} with the energy function
$E$. Notice that  $\omega$ is semibasic and we have the relations
\begin{equation}
i_J\omega=0,\quad
i_C\,\omega=d_JE,\quad\mathcal{L}_C\,\omega=\omega.
\end{equation}
Due to the nondegeneracy of $\omega$, for
any $1-$form $\beta \in \Omega^1(TM)$ there is unique vector field
$\beta^{\#}$ on $TM$ (smooth, in general, only on
$\overset{\circ}{T}M$) such that
\begin{equation}\label{w}
i_{\beta^{\#}}\,\omega=\beta.
\end{equation}
This map $\#:\beta\rightarrow \beta^{\#}$ is called the (Finslerian)
\emph{sharp operator}. In particular, the \emph{gradient} of a
function $f\in C^\infty(TM)$ is the vector field
$\gr\,f:=(df)^{\#}$.

\bigskip

Following Grifone \cite{GRI}, by the \emph{potential} of a semibasic
$k-$form $K$ on $TM$ we mean the $(k-1)-$form $K^\circ:=i_{\s}K$,
where $\s$ is any semispray over $M\,\,$ $(k\geqq 1)$. Clearly,
$K^\circ$ is independent of the choice of $\s$.


To conclude this section, we recall the \emph{fundamental lemma of
Finsler geometry} due to J. Grifone \cite{GRI}.
Let $(M,E)$ be a Finsler manifold. If
\[
\s_0:=-(dE)^{\#} \quad\textrm{over} \quad\overset{\circ}{T}M, \quad
\s_0(0):=0
\]
then $\s_{\,0}$ is a spray over $M$, called the \emph{canonical spray} of
$(M,E)$. $\s_0$ generates a homogeneous Ehresmann connection
$\mathbf{h}_{\,0}$ according to (\ref{eh}), called the \emph{Berwald connection} of
$(M,E)$. $\mathbf{h}_{\,0}$ is conservative in the sense that
\[d_{\mathbf{h}_{\,0}}E=0.\]

\section{Conservative Semibasic Vector 1-Form}
The difference of two Ehresmann connections is a semibasic vector $1-$form,
thus it is natural to define the conservativity of semibasic vector $1-$form as  follows.
\begin{Definition}
\emph{Suppose that $(M,E)$ is a Finsler manifold. A semibasic vector
$1-$form $L$  on $TM$ is said to be \emph{conservative} with respect to energy function E,  if it is a difference of two conservative Ehresmann connections.}
\end{Definition}
It is easy to see that a semibasic vector $1-$form $L$  on $TM$ is conservative if and only if $d_LE=0$. Therefore, the set of all  conservative semibasic vector $1-$forms on $TM$ is a $C^\infty(TM)-$module.

Authors have shown in \cite{SZA} that any homogeneous, conservative Ehresmann connection $\mathbf{h}$ over $M$ can be expressed as follows:
\[\mathbf{h}=\mathbf{h}_{\,0}+\frac12 \mathbf{t}^\circ+\frac12[J,(d_{\,\mathbf{t}^\circ}E)^\#].\]
Hence every homogeneous, conservative Ehresmann connection $\mathbf{h}$ over $M$ generates a conservative semibasic vector $1-$form
$L=\mathbf{t}^\circ+[J,(d_{\,\mathbf{t}^\circ}E)^\#]$.

Another example of conservative semibasic vector $1-$form, is difference of the Berwald and the Wagner Ehresmann connections. Let $f$ be a smooth function on $M$. It is known \cite{VIN2} the Wagner connection  of a Finsler manifold $(M,E)$ have the form,
\[
\overline{\mathbf{h}}:=\mathbf{h}_{\,0}+f^c J - E[J,\gr f^v] - d_JE \otimes\gr f^v,
\]
Then the following semibasic
vector $1-$form is conservative on Finsler manifold $(M,E)$:
\[
L:=f^c J - E[J,\gr f^v] - d_JE \otimes\gr f^v.
\]

\subsection{Conformal Change of Energy Function}\label{uu}
Assume that $(M,E)$ is a Finsler manifold. Let $f$ be a smooth function on $M$ and define a positive function on $TM$ by
\[
\varphi:=exp\circ f^v.
\]
If $\overset{\sim}{E}=\varphi E$, then $(M,\overset{\sim}{E})$ is also a Finsler manifold (see \cite{SV2}).
We say that $(M,\overset{\sim}{E})$ has been obtained by a conformal change of $E$ given by the scale function $\varphi$.
\smallskip

It is known  that a smooth function $\varphi$ on $TM$ is a vertical
lift if and only if $d_J\,\varphi=0$, where $J$ is the canonical vertical endomorphism
on $TM$ (see \cite{SV2}).

Let $K$ be a semibasic vector $1-$form on $TM$. If a smooth function $\varphi$ on $TM$ is a vertical
lift, then
\[\forall\ X\in\mathfrak{X}(TM):\quad (d_K\,\varphi)X\overset{(\ref{d3})}{=}(d\varphi\circ K)X=(KX)\varphi=0, \]
since $KX$ is a vertical vector field. Thus $d_K\,\varphi=0$.

Therefore for the scale function $\varphi:=exp\circ f^v$,  we have $d_K\,\varphi=0$.
\begin{Proposition}
The conservativity property  of a semibasic vector $1-$form on $TM$ (with respect to energy function of Finsler manifold)
is invariant under any conformal change of energy function.
\end{Proposition}
\begin{proof}
Let us consider the conformal change
$\overset{\sim}{E}=\varphi E\ (\varphi:=exp\circ f^v)$ of a Finsler manifold $(M,E)$.
Then
\[
d_{L}\overset{\sim}{E}=d_{L}\,\varphi E
       =(d_{L}\,\varphi) E +\varphi\, d_{L}E
       =\varphi\, d_{L}E.
\]
This relation implies that conservativity of a semibasic vector $1-$form on $TM$ is invariant under any conformal change of
energy function.
\end{proof}
\subsection{L-Ehresmann Connection on a Finsler Manifold}
Let $(M,E)$ be a Finsler manifold and $L$ be a semibasic vector
$1-$form on $TM$. The Ehresmann connection
\[
\mathbf{h}_{L}:=\mathbf{h}_{\,0}+L+[J,(d_LE)^{\#}],
\]
is called \emph{$L-$Ehresmann connection} on Finsler manifold $(M,E)$.
For more details of $L-$Ehresmann connection, we refer to \cite{MBN}.

\smallskip
Next proposition shows that the conservativity of semibasic vector $1-$form $L$ on $TM$ implies the conservativity of $L-$Ehresmann connection.
\begin{Proposition}
Let $(M,E)$ be a Finsler manifold. Suppose that $L$ is a
conservative semibasic vector $1-$form on $TM$ with respect to
energy function $E$. Then the $L-$Ehresmann connection is
conservative.
\end{Proposition}
\begin{proof}
By conservativty of $L$ we obtain $\mathbf{h}_{L}=\mathbf{h}_{\,0}+L$. This
proves  what we want.
\end{proof}

\bigskip

We obtain a condition for conservativity of $L-$Ehresmann connection in the following theorem. First we need to prove next lemma.
\begin{Lemma}\label{lem}
If $\beta$ is a semibasic $1-$form on $TM$, then $\beta^{\#}E=\beta^\circ.$
\end{Lemma}

\begin{proof}
Using definition of canonical spray of Finsler manifold, we get
\begin{align*}
\beta^{\#}E &= dE ( \beta^{\#})= -\big( i_{\,\sa_0}\omega \big)( \beta^{\#} )\\
            &= \omega ( \beta^{\#},\s_0 )= \big( i_{\beta^{\#}}\omega \big)( \s_0 )\\
            &=  \beta\,\s_0=i_{\,\sa_0} \beta\\
            &= \beta^\circ.
\end{align*}
This completes the proof.
\end{proof}

\bigskip

\begin{Theorem}\label{t2}
Let $(M,E)$ be a Finsler manifold and $L$ be a semibasic vector
$1-$form on $TM$. Then $L-$Ehresmann connection is
conservative if and only if $L^\circ E$ is a vertical lift.
\end{Theorem}

\begin{proof}
For simplicity, let us put
\[
U:=(d_{L}E)^\#.
\]
By definition of
$L-$Ehresmann connection  we get
\[
d_{\,\mathbf{h}_{L}}E = d_{\,\mathbf{h}_{\,0}}E + d_{L}E + d_{\,[J,U]}E
\]
The first term of right hand side vanishes by conservativity of the Berwald connection, and for the third one, we have
\begin{align}
d_{\,[J,U]}\,E  &\overset{(\ref{d2})}{=}
i_{\,[J,U]}\,d E \overset{(\ref{ee2})}{=}
i_U\, d_J\, dE+\, d_J\,i_U\, dE
\nr\\
& \overset{(\ref{d2})}{=} -\,i_U\, d \, d_JE +\,d_J\,i_U\, dE
\nr\\
&\,= -\,i_U\,\omega + \,d_J\,i_U\,dE
\nr\\
&\overset{(\ref{w})}{=} - \,d_{L}E + d_J\,i_U\,dE. \nr
\end{align}
Thus, we have
\begin{align*}
d_{\,\mathbf{h}_{L}}E &= d_J\circ i_U\circ dE
\overset{(\ref{d1})}{=} d_J\circ dE\,(U)
\\
&= d_J(U E) \overset{{\textmd{Lem}.}\ref{lem}}{=}
d_J((d_{L}E)^\circ)
\\
&\overset{\dag}{=} d_J (L^\circ E) \nr
\end{align*}
In $\dag$ we have used following relations for an arbitrary semispray $\s$,
\[
(d_{L}E)^\circ = (d_{L}E)(\s) = (i_{L}dE)(\s)
                 = dE(L\s) = dE(L^\circ) = L^\circ E.
\]
The relation
\[d_{\,\mathbf{h}_{L}}E=d_J (L^\circ E)\]
implies that $L-$Ehresmann connection is
conservative if and only if $L^\circ E$ is a vertical lift.
\end{proof}

\bigskip

\begin{Corollary}
Let $(M,E)$ be a Finsler manifold. Suppose that $K$ is a semibasic vector $1-$form on $TM$ and $f$ is a smooth function
on $M$. Define $L:=\frac1{K^\circ E}f^vK$. Then $L-$Ehresmann connection is conservative, provided that the function $K^\circ E$ nowhere vanishes.
\end{Corollary}

\bigskip

\begin{Corollary}
A Wagner connection is conservative.
\end{Corollary}
\begin{proof}
We know that a Wagner connection is an $L-$Ehresmann connection by $L=\frac12\,(f^c J-df^v\otimes C)$ (see \cite{MBN}). Since
\begin{align*}
L^\circ&=\frac12\,(f^c J-df^v\otimes C)^\circ=\frac12\,i_{\s}(f^c J-df^v\otimes C)\\
&=\frac12\,f^c J\s-\frac12(df^v\otimes C)\s=\frac12\,f^c C-\frac12(\s f^v)C\\
&=\frac12\,f^c C-\frac12 f^cC=0
\end{align*}
thus $L^\circ E=0$. Therefore  by Theorem \ref{t2} we obtain the conservativity of Wagner connections.
\end{proof}

\bigskip

Authors have shown in \cite{MBN} that conservativity property of  $L-$Ehresmann connection conformally invariant. Here, we give a simplified proof for this fact by the help of Theorem \ref{t2}.
\begin{Proposition}
The set of all conservative $L-$Ehresmann connections on a Finsler manifold is
conformally closed, i.e., a conservative $L-$Ehresmann connection  remains a conservative $L-$Ehresmann connection
by any conformal change of energy function.
\end{Proposition}

\begin{proof}
Let us consider conformal change $\overset{\sim}{E}=\varphi E$.
Suppose that $\mathbf{h}_{L}$ is conservative on $(M,E)$, then we have
\[
L^\circ\overset{\,\sim}{E}=L^\circ(\varphi E) =(L^\circ
\varphi)E+\varphi\ L^\circ E=\varphi\ L^\circ E.
\]
Since $\varphi$ and $L^\circ E$ are vertical lifts thus
$\overset{\sim}{\mathbf{h}}_{L}$ is conservative on
$(M,\overset{\sim}{E})$ by Theorem \ref{t2}.
\end{proof}
Now, we state and prove a property of $L-$Ehresmann connection for next uses.
\begin{Proposition}\label{p2}
Let $L$ be a semibasic vector $1-$form on $TM$. Then
\[
[C,\Theta_{L}]=\Theta_{[C,L]},
\]
where $\Theta$ is an operator on semibasic vector $1-$forms on $TM$ defined by
\[\Theta:\,L\rightarrow\ \Theta_{L}:=L+[J,(d_LE)^{\#}]=\mathbf{h}_{L}-\mathbf{h}_{\,0}.\]
It results
\[
\mathbf{H}_{L}=\mathbf{h}_{\,0} - \mathbf{h}_{\,[C,L]},
\]
where $\mathbf{H}_{L}$ is the tension of $\mathbf{h}_{L}$.
\end{Proposition}

\begin{proof}
The following holds
\begin{align*}
    \mathcal{L}_C \, d_{L}E= d_{[C,L]}E + d_{L}\mathcal{L}_CE
    = d_{[C,L]}E + d_{L}(2E)= d_{[C,L]+ 2 L}E.
\end{align*}
Using graded Jacobi identity, we obtain
\begin{align*}
    [C,[J,(d_LE)^{\#}]]&=-[J,[(d_LE)^{\#},C]]-[(d_LE)^{\#},[C,J]]
    \\
    &=-[J,-(\mathcal{L}_C d_LE)^{\#}+(d_LE)^{\#}]-[(d_LE)^{\#},-J]
    \\
    &=[J,(\mathcal{L}_C d_LE)^{\#}]-[J,2(d_LE)^{\#}]
    \\
    &=[J,(d_{[C,L]+ 2 L}E)^{\#}]-[J,(d_{2 L}E)^{\#}]
    \\
    &=[J,(d_{[C,L]}E)^{\#}].
\end{align*}
Thus
\begin{align*}
[C,\Theta_{L}]
&=[C,L]+[C,[J,(d_LE)^{\#}]]
=[C,L]+[J,(d_{[C,L]}E)^{\#}]=\Theta_{[C,L]}.
\end{align*}
Then we get the proof.
\end{proof}

\section{Torsion-free Semibasic Vector 1-Form }
By Proposition 1 of \cite{MBN},  the weak torsion of $L-$Ehresmann connection is
\[\mathbf{t}_{L}=[J,L].\]
Due to this result, we define a torsion-free semibasic vector $1-$form as follows.
\begin{Definition}
\emph{A Semibasic vector $1-$form $L$ on $TM$ is called torsion-free if}
\[[J,L]=0.\]
\end{Definition}
Thus a semibasic vector $1-$form $L$ on $TM$ is torsion--free if and
only if the $L-$Ehresmann connection is torsion--free.

Let $V$ be a vertical vector field on $TM$. It is easy to see that the semibasic $1-$form $L:=[J,V]$ is torsion-free.
The following proposition states that the converse is also true. Hence, we get a characterization of torsion-free semibasic vector $1-$forms.

\begin{Proposition}\label{p7}
\label{p3} Let $L$ be a torsion--free semibasic vector $1-$form
on $TM$. Then there is a vertical vector field $V_{L}$ on $TM$
such that
\begin{equation}\label{e1}
L=[J,V_{L}].
\end{equation}
Moreover, the set of all vertical vector fileds on $TM$ satisfying (\ref{e1}) is given by
\begin{equation*}\label{setVL}
\big\{V_L+X^v:\,\, X\in \mathfrak{X}(M)\big\}.
\end{equation*}
\end{Proposition}

\begin{proof}
Since $L$ is torsion-free, there is a semispray $\s$ on $M$ such
that
\[\mathbf{h}_{L}=\frac12\,\big(1_{\mathfrak{X}(TM)}+[J,\s]\big).\]
On the other hand
\begin{align*}
\mathbf{h}_{L}&=\mathbf{h}_{\,0}+L+[J,(d_LE)^{\#}]\\
&= \frac12\,\big(1_{\mathfrak{X}(TM)}+[J,\s_0]\big)+L+[J,(d_LE)^{\#}]\\
&= \frac12\,\big(1_{\mathfrak{X}(TM)}+[J,\s_0+2(d_LE)^{\#}]\big)+L.
\end{align*}
Then
\[
2\,L=[J,\s-\s_0-2(d_LE)^{\#}].
\]
Therefore, it suffices to define
\[
V_{L}:=\frac12\,(\s-\s_0)-(d_LE)^{\#}.
\]
Since for every vector field $X$ on $M$ we have $[J,X^v]=0$, the vertical vector field $V_{L}+X^v$ satisfies (\ref{e1}). Conversely, suppose
that the relation (\ref{e1}) is true for two vertical vector fields. Then
by Lemma 1.15 of \cite{SV1}, the difference of these vertical vector fields is a
vertical lift.
This completes the proof.
\end{proof}

{\bf Remark 1.} Let $L$ be a torsion-free semibasic vector $1-$form on $TM$ and homogeneous of degree $r\neq -1$. By a remark of \cite{GRI} we have $L={1\over r+1}\,[J,L^\circ]$. Hence, the set of all vertical vector fields satisfying (\ref{e1}) is given by
\begin{align*}\label{e3}
\big\{{1\over r+1}\,L^\circ+X^v:\,\, X\in \mathfrak{X}(M)\big\}
\end{align*}

\begin{Lemma}
Let $L$ be a torsion--free semibasic vector $1-$form on $TM$. If
$V_{L}$ is homogeneous of degree 2, then $L-$Ehresmann connection is also homogeneous.
\end{Lemma}
\begin{proof}
By graded Jacobi identity and using Proposition \ref{p7} we obtain
\begin{align}
[C,L]=\big[C,[J,V_{L}]\big] =-\big[V_{L},[C,J]\big]
-\big[J,[V_{L},C]\big]
\overset{(\ref{ee4})}{=}[V_{L},J]+\big[J,V_{L}\big]=0.\nr
\end{align}
Thus $L$ is homogeneous of degree 1, therefore Proposition
\ref{p2}, implies the homogeneity of $L-$Ehresmann connection.
\end{proof}

\bigskip

\begin{Proposition}
Let $L$ be a torsion-free semibasic vector $1-$form on $TM$. Then $d_{\,\mathbf{h}_{L}}\omega$ vanishes, where $\omega$ is
the fundamental form of Finsler manifold $(M,E)$.
\end{Proposition}

\begin{proof}
It is sufficient to show that
\[d_{\,\mathbf{h}_{L}}\omega=-\,d\,d_{\,\mathbf{t}_{L}}E,\]
where $\mathbf{t}_{L}$ is the weak torsion of $L-$Ehresmann connection. First
\[
d_{\,\mathbf{h}_{\,0}}\omega=d_{\,\mathbf{h}_{\,0}}d\,d_JE=-\,d\,d_{\,\mathbf{h}_{\,0}}d_JE
\overset{(\ref{e})}{=}d\,d_Jd_{\,\mathbf{h}_{\,0}}\,E-d\,d_{\,[J,\,\mathbf{h}_{\,0}]}\,E=0.
\]
For simplicity let $P:=[J,(d_LE)^{\#}]$. Then, we have
\begin{align*}
d_P\,\omega&=d_P\,d\,d_JE= -\,d\,d_P\,d_JE
=d\,d_J\,d_PE-\,d\,d_{\,[J,P]}E= d\,d_J\,d_PE.
\end{align*}
Since $[J,P]$ vanishes by graded Jacobi identity. On the
other hand in the proof of Theorem \ref{t2}, we obtain
\[d_PE=-\,d_{L}E+d_J\,i_{(d_LE)^{\#}}\,dE,\]
therefore
\[
d_P\,\omega=-\,d\,d_J\,d_{L}E+d\,d_J^{\,2}\,i_{(d_LE)^{\#}}\,dE
=-\,d\,d_J\,d_{L}E.
\]
Since
$d_J^{\,2}\overset{(\ref{e})}{=}{1\over2}d_{\,[J,J]}\overset{(\ref{ee4})}{=}0$.
Hence
\begin{align*}
d_{\,\mathbf{h}_{L}}\omega&=d_{\,\mathbf{h}_{\,0}}\omega+\,d_{L}\omega+\,d_P\,\omega=d_{L}\,d\,d_JE-\,d\,d_J\,d_{L}E
\\
&=-\,d\,d_{L}\,d_JE-\,d\,d_J\,d_{L}E= -\,d\,d_{\,[J,L]}E\\
&=-\,d\,d_{\,\mathbf{t}_{L}}E.
\end{align*}
The proof follows from the torsion--freeness of $L$.
\end{proof}
\subsection{Some Projectively Related Sprays}
In this section we construct a class of semisprays that generates the torsion-free $L-$Ehresmann connections.
\begin{Proposition}\label{THM2}
Let $V$ be a vertical vector field on $TM$, then the Ehresmann connection  generated by semispray
\begin{equation}\label{new2}
\overset{\,\,V}{\s}:=\s_0+2\,V+2\,(d_{\,[J,V]}E)^\#.
\end{equation}
is just the $[J,V]-$Ehresmann connection.
Furthermore, if $V$ is a vertical vector field homogeneous of degree 2, then the semispray $\overset{V}{\s}$ is a spray.
\end{Proposition}
\begin{proof}
By a direct computation, we get
\begin{align*}
\frac12\,\big(1_{\mathfrak{X}(TM)}+[J,\overset{\,\,V}{\s}]\big)
&=\frac12\,\big(1_{\mathfrak{X}(TM)}
+[J,\s_0+2V+2(d_{\,[J,V]}E)^\#]\big)
\\
&=\mathbf{h}_{\,0}+[J,V]+[J,(d_{\,[J,V]}E)^\#]\\
&=\mathbf{h}_{\,[J,V]}.
\end{align*}
Suppose $V$ is a $2-$homogeneous vertical vector field. By Jacobi identity, we have
\begin{align*}
[C,[J,V]]&=[V,[J,C]+[J,[C,V]]=[V,J]+[J,V]=0.
\end{align*}
Therefore $[J,V]$ is homogeneous of degree 1. Using Lemma 2 of \cite{MBN},
we obtain that $V+2\,(d_{\,[J,V]}E)^\#$ is homogeneous of degree 2.
Hence the semispray $\overset{V}{\s}$ is a spray.
\end{proof}

By Proposition \ref{THM2}, we get the following.

\begin{Corollary}\label{c1}

Suppose that $L$ is a torsion--free semibasic vector $1-$form on
$TM$. Then $\overset{\ V_{L}}{\s}$ generates $L-$Ehresmann connection.
\end{Corollary}
\bigskip

Two sprays $\s_1$ and $\s_2$ over $M$ are said to be
(\emph{pointwise}) \emph{projectively related} if there is a smooth
function $\lambda$ on $\overset{\circ}{T}M$ such that
$$\s_2=\s_1+\lambda C,\quad (\textmd{over}\ \overset{\circ}{T}M).$$
Then the \emph{projective
factor} $\lambda$ is necessarily $1-$homogeneous, i.e., $C\lambda=\lambda$. For more details see \cite{SVa}.

\begin{Proposition}
Let $V$ and $U$  be two vertical vector fields homogeneous of degree 2 on $TM$.
Suppose that $\overset{V}{\s}$ and $\overset{U}{\s}$ are
projectively related sprays on $M$, with the projective factor $\lambda$. Then
$$\lambda=\frac{3(V-U)E}{E},\quad (over\ \overset{\circ}{T}M).$$
\end{Proposition}
\begin{proof}
Suppose that $\overset{V}{\s}$ and $\overset{U}{\s}$ are
projectively related sprays on $M$, with the projective factor $\lambda$. Then
\[\overset{V}{\s}=\overset{U}{\s}+\lambda C.\]
Therefore by (\ref{new2}) we get
\[\s_{\,0}+2\,V+2\,(d_{\,[J,V]}E)^\#=\s_{\,0}+2\,U+2\,(d_{\,[J,U]}E)^\#+\lambda C.\]
Put $Y:=2\,(V-U)$ then we obtain
\[
Y+(d_{[J,Y]}E)^{\#}=\lambda C.
\]
We have
\[i_{\lambda C}\,\omega=\lambda\,i_C \omega=\lambda\,d_JE,\]
on the other hand

\[
i_{Y}\,\omega+i_{(d_{[J,Y]}E)^{\#}}\,\omega=i_{Y}\,\omega+\,d_{[J,Y]}E.
\]
Hence, for any semispray $\s$,
\begin{align*}
(i_{Y}\omega)(\s)&=\omega\,(Y,\s)=(d\,d_JE)(Y,\s)
\\
&=Y(d_JE)(\s)-\s(d_JE)(Y)-(d_JE)([Y,\s])
\\
&=2YE-YE=YE,
\end{align*}
and
\begin{align*}
(d_{[J,Y]}E)(\s) &= dE([J,Y]\s)=dE([C,Y]-J[\s,Y])= 2\,dE(Y)= 2\,YE,
\end{align*}
and
\begin{eqnarray*}
\lambda (d_JE)(\s)=\lambda dE(J\s)=\lambda dE(C)=\lambda CE=2\lambda E,
\end{eqnarray*}
thus
$3\,YE=2\lambda E.$
Therefore
\[
3\,(V-U)E=\lambda E,
\]
as we want.
\end{proof}
\section{Some Conservative Vector Fields}
In \cite{VIN},  Vincze  defines a conservative  vector field
on a Finsler manifold. He proves that a vertical vector field  $V$ is conservative on a Finsler manifold $(M,E)$ if
and only if
\[i_V\omega=d_J(VE).\]
In this section,  we introduce a class of conservative vertical vector fields.
\begin{Theorem}\label{t3}
Let  $(M,E)$ be a Finsler manifold and $V$ be a vertical vector
field on $TM$. If $[J,V]-$Ehresmann connection is conservative
then  $U:=V+(d_{[J,V]}E)^\#$ is a conservative vector field.
\end{Theorem}

\begin{proof}
By assumption we obtain
\[\mathbf{h}_{[J,V]}=\mathbf{h}_{\,0}+[J,V]+[J,(d_{[J,V]}E)^\#]=\mathbf{h}_{\,0}+[J,U].\]
Suppose that $[J,V]-$Ehresmann connection is conservative. By
conservativity of $\mathbf{h}_{\,0}$ and using Cartan magic formula we have
\begin{align}
0&=d_{\mathbf{h}_{\,[J,V]}}E=d_{\mathbf{h}_{\,0}}E+d_{[J,U]}E
=d_J\mathcal{L}_UE-\mathcal{L}_Ud_JE
\nr\\[1.3ex]\nr
&=d_J(UE)-i_Udd_JE-d\, i_Ud_JE =d_J(UE)-i_U\omega,
\end{align}
which implies that $U$ is a conservative vector field.
\end{proof}

\bigskip

\begin{Corollary}
Let  $(M,E)$ be a Finsler manifold and $L$ be a torsion-free semibasic vector
$1-$form on $TM$. If $L-$Ehresmann connection is
conservative  then
\[V_{L}+(d_LE)^{\#},\]
is conservative in Vincze's sense.
\end{Corollary}


\bigskip
\noindent
Akbar Tayebi  and Mansoor Barzegari\\
Faculty  of Science, Department of Mathematics\\
Qom University\\
Qom. Iran\\
Email:\ mbarzegari@ymail.com\\
Email:\ akbar.tayebi@gmail.com

\end{document}